\pdfoutput=1
\documentclass[nohyperref]{article}

\usepackage{microtype}
\usepackage{graphicx}
\usepackage{subfigure}
\usepackage{booktabs} 

\usepackage{hyperref}

\usepackage[accepted]{icml2022}

\usepackage{amsmath}
\usepackage{amssymb}
\usepackage{mathtools}
\usepackage{amsthm}
\usepackage{enumitem}

\usepackage[capitalize,noabbrev]{cleveref}

\theoremstyle{plain}
\newtheorem{theorem}{Theorem}[section]

\newtheorem{lemma}[theorem]{Lemma}
\newtheorem{corollary}[theorem]{Corollary}
\theoremstyle{definition}
\newtheorem{definition}[theorem]{Definition}
\newtheorem{example}[theorem]{Example}
\newtheorem{assumption}[theorem]{Assumption}
\theoremstyle{remark}

\newcommand{\E}{\mathbb{E}}

\newcommand{\grad}{\nabla}

\newcommand{\tsum}{\textstyle\sum} 

\newcommand{\iprod}[2]{\langle #1, #2 \rangle}

\newcommand{\abs}[1]{\left|{#1}\right|}

\newcommand{\cO}{\mathcal{O}}

\newcommand{\N}{\mathbb{N}}

\newcommand{\Var}{\text{Var}}
\newcommand{\F}{\mathcal{F}}

\icmltitlerunning{Moment Estimation Improves Instance Complexity}

\DeclarePairedDelimiterX{\inp}[2]{\langle}{\rangle}{#1, #2}
\DeclarePairedDelimiterX{\cbr}[1]{\{}{\}}{#1} 
\DeclarePairedDelimiterX{\rbr}[1]{(}{)}{#1} 
\DeclarePairedDelimiterX{\sbr}[1]{[}{]}{#1}

\begin{document}

\twocolumn[
\icmltitle{ Beyond Worst-Case Analysis in Stochastic Approximation:\\ Moment Estimation Improves Instance Complexity}



\icmlsetsymbol{equal}{*}

\begin{icmlauthorlist}
\icmlauthor{Jingzhao Zhang}{thu}
\icmlauthor{Hongzhou Lin}{am}
\icmlauthor{Subhro Das}{ibm}
\icmlauthor{Suvrit Sra}{MIT}
\icmlauthor{Ali Jadbabaie}{MIT}

\end{icmlauthorlist}

\icmlaffiliation{thu}{IIIS, Tsinghua University}
\icmlaffiliation{MIT}{Massachusetts Institute of Technology}
\icmlaffiliation{am}{Amazon}
\icmlaffiliation{ibm}{MIT-IBM Watson AI Lab, IBM Research}

\icmlcorrespondingauthor{Jingzhao Zhang}{jingzhaoz@mail.tsinghua.edu.cn}

\icmlkeywords{Stochastic approximation, gradient methods, optimization}

\vskip 0.3in
]



\printAffiliationsAndNotice{}

\begin{abstract}
We study oracle complexity of gradient based methods for stochastic approximation problems. Though in many settings optimal algorithms and tight lower bounds are known for such problems, these optimal algorithms do not achieve the best performance when used in practice. We address this theory-practice gap by focusing on \emph{instance-dependent complexity} instead of  worst case complexity. In particular, we first summarize known instance-dependent complexity results and categorize them into three levels. We identify the domination relation between different levels and propose a fourth instance-dependent bound that dominates existing ones. We then provide a sufficient condition according to which an adaptive algorithm with moment estimation can achieve the proposed bound without knowledge of noise levels. Our proposed algorithm and its analysis provide a theoretical justification for the success of moment estimation as it achieves improved instance complexity.

\end{abstract}


\section{Introduction}

Stochastic approximation (SA) methods, introduced and analyzed in the seminal works~\cite{Robbins1951,Fabian1968, Polyak1987, Kushner2003, Moulines2011}, are central to modern machine learning algorithms. Due to their wide applicability and importance, they have been extensively studied. It is now well established that an algorithm as  simple as stochastic gradient descent achieves  minimax optimal rates in many  settings~\cite{Nesterov2013, Agarwal2009, Fang2018, Arjevani2019, Lin2015,AllenZhu2016b, Ghadimi2012}.



These well-studied stochastic methods are, however,  optimal in the sense of a worst case measure of complexity:
\begin{align}\label{eq:complexity}
	\inf_{A_\theta \in \mathcal{A}} \sup_{\substack{x_1 \in \mathbb{R}^d , f \in  \mathcal{F}}} T_\epsilon(A_\theta[f, x_1], f),
\end{align}
where $T_\epsilon(A_\theta[f, x_1], f)$ denotes the number of iterations to obtain an $\epsilon$-solution initialising at~$x_1$. In the most common setting of convex problems, the $\epsilon$-solution is measured by the function suboptimality gap:
\begin{align}\label{eq:T-det}
	T_\epsilon(\{x_t\}_{t \in \mathbb{N}}, f) := \inf \cbr*{t \in \mathbb{N} | f(x_t) - f(x^*) \le \epsilon}.
\end{align}
The worst case complexity measure is naturally achieved in certain extreme cases within a function class, which may be quite rare in practice. Empirical observations~\cite{Defazio2018, Reddi2019} suggests that a gap does exist between worst case analysis and  empirical performance.

To reduce the mistmatch between theory and practice, it is necessary to go beyond worst case scenarios, and to include more domain-specific nuances. One alternative is to consider average case complexity analysis. Such an approach is well established in theoretical computer science,  and is used to explain the superior empirical performance of QuickSort~\cite{hoare1962quicksort}, quickhull~\cite{preparata2012computational}, simplex methods~\cite{borgwardt2012simplex}. In contrast, the use of average-case complexity analysis in optimization scenarios is less mature and in its infancy~\cite{pedregosa2020average, lacotte2020optimal,paquette2021sgd}.

Another approach to move away from the worst-case is to conduct smoothed analysis~\cite{spielman2005smoothed}, in which the complexity is instance-dependent~\cite{fagin2003optimal, afshani2017instance}. Here, instance could be referring to a specific sample, or more broadly to an easily parametrizable subclass of samples. The goal is to build a middle ground between worst-case and average-case analyses, such that we can eventually interpolate between them~\cite{spielman2005smoothed}. This idea has been recently extended to the study of complexity in Q-learning~\cite{khamaru2021instance,pananjady2020instance}, and further analyzed in Markovian linear stochastic approximation setting~\cite{mou2021optimal}. 

Following this line of work, we propose an \emph{instance-dependent analysis} in the general convex setup, extending the linear setting and worst-case optimal rates in~\cite{mou2021optimal}. We define instances by the subclass of functions with non-stationary noise, and make the following contributions:
\begin{itemize}
\setlength{\itemsep}{1pt}
	\item We categorize previous instance complexities into three categories based on their dependence on the iteration-wise noise levels as well as the information available to the algorithm. We further identify the gap between different types of instance complexities.
	\item We propose a new instance dependent rate, named \emph{dynamic error bound},  that dominates  (see Definition~\ref{def:order}) previous ones when additional information of the problem instance is available.
	\item We study an algorithm based on moment estimation that achieves the \emph{dynamic error bound} without knowing additional information when the total variation in noise level is bounded. We hence provide the first theoretical justification for why moment estimation speeds up gradient methods empirically.
\end{itemize}

We test the noise variation and empirical performance of our moment estimation algorithm on policy optimization and neural network training. We conclude by discussing numerous open problems that this work raises.

\section{Problem Setup and Instance Complexity}

\begin{figure}[t]
	\centering
	{
		\includegraphics[width=0.9\columnwidth]{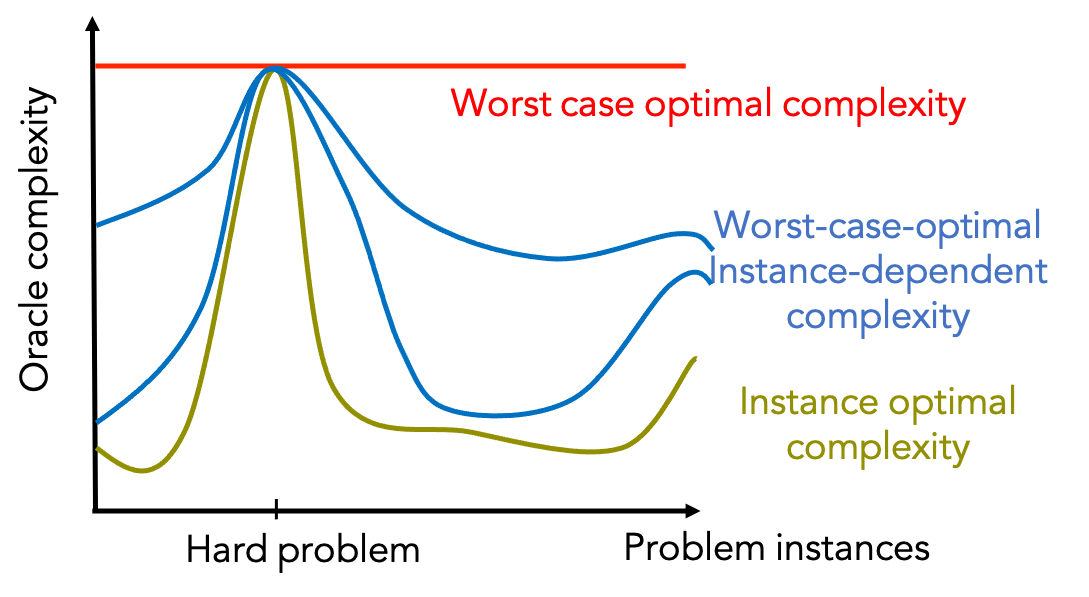}
	} 
	\caption{Different notions of  complexity and optimality.}
	\label{fig:complexity}
\end{figure}

In this work, we focus on stochastic approximation problems within the convex smooth function class $\mathcal{F}_{L, R}$:

\begin{assumption}
	A continuously differentiable function $f$ is in $\mathcal{F}_{L, R}$ if $f$ is convex, $L-$smooth, and bounded below, 
	\begin{align*}
		&f(y) \ge f(x) + \iprod{\grad f(x)}{y-x},\\
		&\|x_1 - x^*\| \le R,\\
		&\|\grad f(x) - \grad f(y)\| \le L\| x - y \|,
	\end{align*}
where $x^*$ denotes a global optimal solution.
\end{assumption}


We follow the classical stochastic approximation setup (\emph{\`a la} Robbins-Monro), where {\bf the noise is additive}. More precisely, the stochastic gradient oracle $g$ is given as the sum of the actual gradient $\nabla f$ and random noise~$\xi$:
\begin{equation}
    g(x) =\nabla f(x) + \xi.
\end{equation}
This setup holds whenever the error comes from an exogenous source. In pursuit of our goal to go beyond worst case analysis, we follow the recent line of work on instance-complexity for stochastic approximation and reinforcement learning~\cite{khamaru2021instance,pananjady2020instance,mou2021optimal}, and define the  $\{ \sigma_k \}_{k \in \N}$-instance as the following subclass of gradient oracles:

\begin{definition}[$\{ \sigma_k \}_{k \in \N}$-instance]
At the $k^{\rm th}$ iteration, the gradient oracle $g_k$ returns a stochastic gradient 
\[ g_k(x) = \nabla f(x) + \xi_k,\] 
where $\xi_k$ is zero mean and has bounded second moment $\sigma_k$:
	$$\E[\xi_k] = 0, \quad  \E[\|\xi_k \|^2] = \sigma_k^2 \le M^2.$$
\end{definition}

For simplicity we assume that the noise is nontrivial:
\begin{assumption}[Nontrivial noise]
\label{as:nontriv}
		$$  \sigma_k^2>  L^2 R^2 / T ,\  \text{for all } k \le T. $$
\end{assumption}
Indeed, when the noise level is too small, $g_k$ is essentially the full gradient $\nabla f$, leading to a deterministic style convergence. Assumption~\ref{as:nontriv} is not necessary but saves us from discussing such trivial cases, which deviate from the main purpose of studying instance-dependent bounds.

Unlike standard analysis where uniform variance bound is imposed on all the stochastic gradients, we allow the noise level to be non-stationary, i.e., can vary from iteration to iteration. To recover the worst case analysis, it suffices to set the noise at a maximum level, $\forall k, \sigma_k =M$. By allowing the noise to be non-stationary, we create a fine-grained middle ground where we can investigate how the variation of noise $\sigma_k$ influence the  convergence rate of an optimization algorithm.

As in most convex problems,  we study the level of suboptimality achieved after querying the oracle $T$ times to upper bound $\E[f(x_T) - \min_x f(x)]$. In the following subsection, we derive \emph{theoretical convergence rates under a $\{\sigma_k\}_{k \in \mathbb{N}}$-instance, named as the instance complexity}.



\subsection{Different levels of instance complexity}
\newcommand{\ra}[1]{\renewcommand{\arraystretch}{#1}}
\begin{table*} [htp!]
	\ra{3}
	\centering
	\resizebox{\textwidth}{!}{
		\begin{tabular}{@{}lcrcrcrcrc@{}}\toprule
			& {Worst} & \phantom{a}& {Agnostic} & \phantom{a} & {Adaptive} &
			\phantom{a} & {Dynamic} &  \phantom{a}  \\ \midrule
			{Error bound} 
			& $\tfrac{2RM}{\sqrt{T}}$  && $ (R^2 + \tfrac{1}{T} \sum_k \sigma_k^2 ) / \sqrt{T}$  && $ 2R  \left(\tfrac{1}{T} \textstyle\sum\limits_{k=1}^T \sigma_k^2 \right)^{1/2} / \sqrt{T} $ && 2R $\left(\tfrac{1}{T} \sum\limits_{k=1}^T \tfrac{1}{\sigma_k}\right)^{-1} / \sqrt{T}$ \\ 
			
			\midrule
			
			{$\eta_k$} 
			& $R / \sqrt{TM^2}$  && $ 1/ \sqrt{T}$  && $ R / \sqrt{\textstyle\sum_{k=1}^{T} \sigma_t^2}$\ or\ $R/ \sqrt{2 \sum_{\tau \le k} \|g_k\|^2}$ && $ R / (\sigma_k\sqrt{T})  $ \\
			
			\midrule
			
			\shortstack{Can be \\ achieved via} 
			& \shortstack{Fixed step, \\ known $R, M$ } && \shortstack{Fixed step, \\ unknown $R, M$ }  && \shortstack{Fixed step, known $R, \{\sigma_k\}_k$ or \\ Adapt. step, unknown $\{\sigma_k\}_k$ } && \shortstack{ Adaptive step,\\ known $R, \{\sigma_k\}_k$ } \\
			\bottomrule
	\end{tabular}}
	\caption{Comparisons of four types of instance-dependent theoretical convergence rates under the noise level $\{\sigma_k\}_k$. The parameter $R^2$ denotes  the initial distance $\| x_1-x^* \|^2$ and the parameter $M$ is an upper bound of $\sigma_k$. The convergence rate improves from the left to the right monotonically, with different choices of stepsizes $\eta_k$. }\label{tab:complexity}

\end{table*}

In this subsection, we categorize rates in stochastic approximation into four levels based on their dependence on the problem instance defined by $\{\sigma_i\}_{i \in \mathbb{N}}$. \textbf{Our proposed categorization later enables us to identify a nontrivial gap between vanilla stochastic gradient descent and its adaptive variant using moment estimation}. Although we focus on analysis for convex functions, similar categorization may also hold for analyses of nonconvex functions.

We begin by summarizing known results that are dependent on iteration-wise noise levels $\{\sigma_k\}_k$. Previous results on instance-dependent bounds usually do not consider the exact same problem setup, and hence not directly comparable. To have a unified framework, we view these results under the smooth-convex setting. Recall the following folklore result:



\begin{theorem}\label{thm:baselines}
When $f \in \mathcal{F}_{L, R}$ (convex, $L$-smooth), the stochastic gradient descent algorithm of the form
	$$x_{k+1} = x_k - \eta_k g(x_k), $$
	with predetermined stepsizes { ${\eta_k \le \frac{1}{L}}$}, 
	satisfies the suboptimality bound
	\begin{equation}\label{sgd convergence}
		\E[f(\overline{x}_T) -f^*] \le \frac{R^2 + \tsum_{k=1}^T \eta_k^2 \sigma_k^2}{ \tsum_{k=1}^T \eta_k},
	\end{equation}
where $\overline{x}_T = (\tsum_{k=1}^T \eta_k x_k) / (\tsum_{k=1}^T \eta_k)$ is the weighted average of iterates.
\end{theorem}

The above theorem already shows that different step size choices can induce a variety of instance complexities, and some choices might give faster convergence. However, better step size choice might require more information about the problem instance. We will now discuss how different step size choice requires different level of information and classify previous work based on their step size strategy.

\paragraph{1. Worst case bounds.} In classical stochastic optimization literature (e.g. ~\cite{Rakhlin2011,Ghadimi2012,Ghadimi2013a, Fang2018}), only an upper bound on the variances is given. In other words, we do not have access to the entire sequence of $\sigma_k$ but only an upper bound  $M \ge \sigma_k, \forall k$. Using the stepsize $ \eta_k = R / \sqrt{TM^2}$, this choice results in the worst case bound:
\begin{align} \label{eq:worst}
		\E[f(\overline{x}_T) -f^*] \le  2RM / \sqrt{T} := \epsilon_{\text{worst}}.
\end{align}

\paragraph{2. Agnostic bounds.} Another type of bound appears in the analysis of adaptive methods, and involves an additive bound on the initialization ($\|x_0 - x^*\|, f(x_0) - f(x^*)$) and noise (see Theorem 2.1 of \cite{Ward2018}, or \cite{Zhou2018c}) instead of a multiplicative one as in other bounds. 

Though these bounds are usually achieved by adaptive methods, to distinguish them from other instance-dependent bounds, we call these bounds  ``\emph{agnostic bounds}'' because they can be obtained using a step size without knowing the maximum noise level, the smoothness constant or the domain diameter of the problem. In smooth convex case, these bounds can be achieved by setting $\eta_k = \frac{1}{\sqrt{T}}$:
\begin{align} \label{eq:agnostic}
	\E[f(\overline{x}_T) -f^*] \le (R^2 + \tfrac{1}{T} \tsum\limits_{k=1}^T \sigma_k^2 ) / \sqrt{T} := \epsilon_{\text{agnostic}}.
\end{align}

\begin{table*}[ht!]
	\centering 
	\ra{1}
	\resizebox{0.7\textwidth}{!}{\begin{tabular}{@{}lcrcrcrcrc@{}}\toprule
			& {Worst} & \phantom{a}& {Agnostic} &
			\phantom{a} & {Adaptive} &  \phantom{a} & {Dynamic} & \phantom{a} \\ \midrule
			{Error bound} 
			& $T^{-\frac{1}{2}}$ &&$T^{-\frac{1}{2}}$    && $T^{-\left ((1 + \alpha ) / 2 \right )}$ && $T^{-\left ((1 + 2\alpha) / 2  \right )}$ \\
			\bottomrule
	\end{tabular}}
	\caption{Comparison of convergence rate under multiple shapes of the noise level $\sigma_k$, where constant terms are omitted. The parameter $\alpha \in[0,\frac{1}{2}]$ is the exponent controlling the ratio between $\max \sigma_k$ and $\min \sigma_k$, indicating the significance in noise changes. }\label{tab:example_theory}
\end{table*}

\paragraph{3. Adaptive bounds.} One theoretically justified advantage of adaptive gradient methods is that if the smoothness and level of suboptimality are known, then adaptive methods can automatically adjust the learning rate for the noise level in stochastic gradients. Results of this type can be found for example in~\cite{ Duchi2011, Levy2018, Ward2018, Zhou2018c}. In particular, the setup in~\cite{Levy2018} (Thm 2.1 and equation (11)) subsumes our smooth-convex setup and is directly comparable. In particular, by setting $\eta_k = R(2 \sum_{\tau \le k} \|g_k\|^2)^{-1/2}$, one gets the following rate: 
\begin{align} \label{eq:adaptive}
	 \E[f(\overline{x}_T) -f^*] \le  \frac{2R}{\sqrt{T}}  \left(\tfrac{1}{T} \textstyle\sum\limits_{k=1}^T \sigma_k^2 \right)^{\frac{1}{2}} := \epsilon_{\text{adaptive}} .
\end{align}
Notice that results in~\cite{mou2021optimal} also fall in this category. Though the bounds in~\cite{mou2021optimal}  may not be directly comparable as they are in the strongly convex regime, these bounds, as well as the adaptive bound in~\eqref{eq:adaptive} can be achieved by using a fixed step size if the noise levels $\sigma_k$ are known in advance by setting $\eta_k = R \left(\textstyle\sum_{\tau=1}^{T} \sigma_\tau^2\right)^{-1/2}$. This is the best attainable rate using a fixed step size by optimizing the error bound in Theorem~\ref{thm:baselines}.

This observation naturally leads to the bounds in the next part, which are achieved by an iteration dependent step size when  $\sigma_k$ are known. From the following result, we will see that the adaptive bounds above are \textbf{worst-case optimal} but \textbf{not instance-level optimal} (see Figure~\ref{fig:complexity}).

\paragraph{4. Dynamic bounds.} 
A natural question is whether we can further improve adaptive bounds by appropriately selecting the stepsizes. This is indeed possible by setting a stepsize inversely proportional to $\sigma_k$.  More concretely, when setting $\eta_k = R / (\sigma_k\sqrt{T})$, we get
\begin{align} \label{eq:dynamic}
	\E[f(\overline{x}_T) -f^*]  \le 2R \left(\tfrac{1}{T} \tsum\limits_{k=1}^T \tfrac{1}{\sigma_k}\right)^{-1} / \sqrt{T} := \epsilon_{\text{dynamic}}.
\end{align}
We name these bounds ``\emph{dynamic bounds}'', analogous to ``dynamic regret'' in online learning~\cite{Jadbabaie2015,yang2016tracking}, indicating that the optimal policy can have iteration dependent action (i.e. iteration-wise step sizes) rather than a fixed action (i.e. a fixed step size).

Such bounds are less studied, as previous works mainly focus on the dependence of $T$ instead of the fine-grained dependency on the instance~ $\{\sigma_k\}_k$. In our setting, it is important to notice that different dependency on $\{\sigma_k\}_k$ could lead to non-trivial differences in the convergence rate. For instance, the dynamic bound depends on the harmonic sum of $\sigma_k$, whereas the adaptive bound depends on its $2$-norm. To make comparison between different error bounds, we adopt the following notation.

\begin{definition} [Dominating bounds]\label{def:order}
	We say an error bound dominates another, denoted as $\epsilon_1 \preceq \epsilon_2$,  if there exists an absolute constant $c > 0$ such that for any $R, \{\sigma_k\}_k, T$, we have that the instantiated value $\epsilon_1(R, T,\{\sigma_k\}_k ) \le c \epsilon_2(R, T,\{\sigma_k\}_k )$. 
\end{definition}

With the above definition, we could provide the following order of instance complexities.

\begin{lemma}\label{lemma:ordering}
	The four different types of bounds satisfy the following ordering:
	\begin{align*}
	  \epsilon_{\text{dynamic}} \preceq \epsilon_{\text{adaptive}}  \preceq   \min\{\epsilon_{\text{worst}} ,	\epsilon_{\text{agnostic}}  \}.
	\end{align*}
\end{lemma}

The above result is expected. We see in Table~\ref{tab:complexity} that the dominance between different bounds is ordered the same way the amount of information available to the algorithm. 

Another straightforward consequence of the lemma is that both $\epsilon_{\text{adaptive}}$ and $\epsilon_{\text{dynamic}}$ are worst case optimal. In other words, they all fall into the same rate when $\sigma_k = M$ is constant. Similarly, when all the ratios $\sigma_i/ \sigma_j$ are bounded by a constant, then the ratio of these rates are bounded, i.e. differ only by a constant ratio.

A more interesting question to ask is when does the adaptive rate $\epsilon_{\text{adaptive}}$ or the dynamic rate $\epsilon_{\text{dynamic}}$  non trivially improve the worst case rate. This can happen when the variation in the noise level $\{\sigma_k\}_k$ is comparable to the total iteration number $T$. In fact,  \textbf{we will later show that gradient methods with moment estimation can achieve such bounds without knowing the noise level in advance, and hence improve over vanilla SGD. } To make the discussion concrete, we first use a synthetic example to illustrate the phenomenon.

\subsection{Complexity comparison via synthetic example}

In this section, we use an example with synthetic noise to illustrate the instance complexity of different types.


\begin{figure} [h]
	\centering
	{
		\includegraphics[width=0.7\columnwidth]{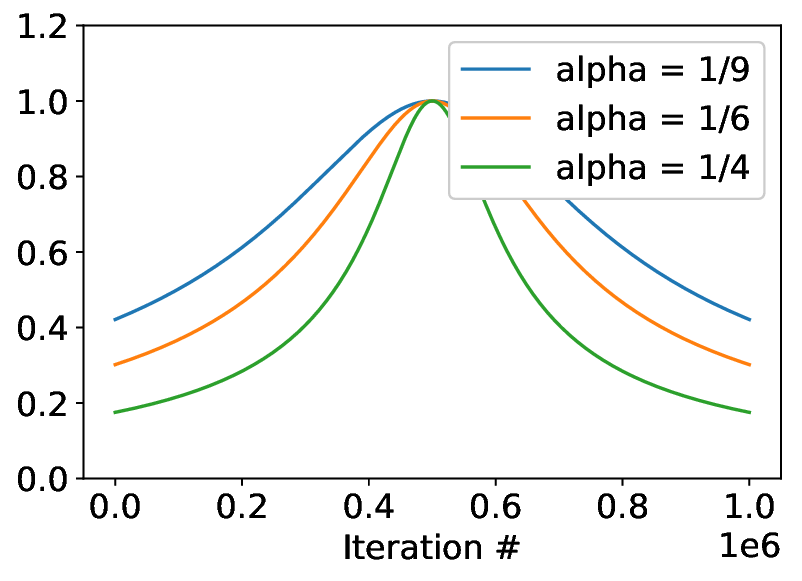}
	}  
	\caption{Mountain shape noise for different values of $\alpha$.}
	\label{fig:example}
\end{figure}

\begin{example}\label{example2} 
	We consider a noise sequence $\sigma_k$ of a mountain shape, illustrated in Figure~\ref{fig:example}.  The noise level increases  smoothly in the first half of iterations, then decreases in the second half, parametrized by a positive constant $\alpha \in (0, 1)$:
	\begin{align*}
		\sigma_k = \tfrac{1}{\sqrt{1+ T^{2\alpha} \left (\frac{2k}{T} - 1 \right )^2}}.
	\end{align*}
\end{example}
In particular, the maximum of $\sigma_k$ is $1$, attained in the middle of the iterations; and the minimum of it is $T^{-\alpha}$. Substituting the values of $\sigma_k$ into the rates summarized in Table~\ref{tab:complexity}, we obtain the rates shown in Table~\ref{tab:example_theory}.

For this specific $\{ \sigma_k\}$-instance, adapting the noise level is beneficial. As we can see, the adaptive bound improves the worst case bound by $T^{\alpha/2}$, and, the dynamic bound improves further the adaptive bound by $T^{\alpha/2}$. Such a non-trivial gap is unobservable under worst case analysis. 


Though dynamic error bounds can achieve faster convergence than standard worst case rates (in terms of $T$), we note that achieving better dependence on $T$ is not the main purpose of our work. \emph{Instead, we want to show via different exponents on $T$ that the gap between different types of instance complexity is significant and cannot be considered as off by an absolute constant.}


We showed that using a step size that is inversely proportional to the noise level can generate a better instance-dependent bound. Such a construction relies on the knowledge of $\sigma_k$. However, in practice, we usually do not have access to the exact noise level. We show in the next section that combining the idea of moment estimation with SGD can achieve the dynamic error bound even when the $\{\sigma_k \}_k$ are unknown.

\section{Moment Estimation Achieves Dynamic Error Bounds}


The main difficulty to achieve the dynamic bounds lies in the estimation of $\sigma_k$. A naive first idea is to draw multiple samples of the gradient at each iteration and perform an empirical estimate of the variance. Later use the estimate $\widehat{\sigma}_k$ to update the stepsizes. In order to ensure nonasymptotic concentration of the estimator, one convenient way is to impose a bounded higher moment assumption. 

\begin{assumption}[Higher moment boundedness.]\label{assump:fourth moment}
	We assume that the fourth moment of $g_k$ is bounded, namely, $\E[\| \xi_k\|^4] = \E[\|g_k -\nabla f \|^4] \le M^4$ for all $k$.
\end{assumption}

This assumption allows us to guarantee the estimation of $\sigma_k^2$ is non-asymptotically concentrated with high probability. In particular, given $N$ samples of the gradients $g_{k,1}, g_{k,2}, \cdots g_{k,N}$ at the $k$-th iteration, the standard variance estimator  
\begin{align}\label{eq:ub-var-estimator}
	Y_k  &= \tfrac{\sum_{i=1}^N \|g_{k,i} - \overline{g}_k \|^2 }{N-1},  
	 \text{ where }  \overline{g}_k = \tfrac{1}{N} \textstyle\sum_{i=1}^N g_{k,i},
\end{align}
is an unbiased, i.e. $\E[Y_k] = \sigma_k^2$. Moreover, its variance is  bounded: $\Var[Y_k^2]  =  \tfrac{1}{N} \left ( \E[\|\xi_k\|^4] - \tfrac{N-3}{N-1} \sigma_k^4 \right ) \le \tfrac{M^4}{N}.$

In other words, with high probability, the empirical estimator $Y_k$ concentrates around $\sigma_k^2 \pm \tfrac{M^2}{\sqrt{N}}$. However, if $\sigma_k = O(T^{-\alpha})$, we would need $T^{4\alpha}$ samples to make sure the approximation is in the same magnitude of $\sigma_k$. This leads to an extremely inefficient sample complexity, barely implementable in practice. 


In the real-world setting, often the environment changes in a smooth way. This provides another possibility if we are able to exploit the continuity in the noise changes. In particular, we can combine the previous gradients with the actual ones to perform the variance estimation, i.e. conducting a moving average estimator. In order to provide theoretical guarantee, we impose the following condition on the total variation of noise level~$\sigma_k^2$.

\begin{assumption}[Total variation boundedness.]\label{assump:sum}
	We assume that the total variation of $\sigma_k^2$ is bounded, i.e., $\sum\nolimits_k |\sigma_k^2 - \sigma_{k+1}^2| \le D^2$,
	where $D$ is a constant independent of $T$. To facilitate discussion, we assume $D^2 = \Omega(M^2)$.
\end{assumption}

The bounded total variation assumption provides us the necessary smoothness on $\sigma_k$. This assumption is commonly used in the dynamic, online learning literature \citep{Besbes2014, Jadbabaie2015, Mokhtari2016}. A key consequence of this assumption is that it  avoids infinite oscillations, such as the pathological setting where $\sigma_{2k}=\frac{1}{T^\alpha}$ and $\sigma_{2k+1}=1$, in which case the total variation scales with the number of iterations~$T$. The specific constant in $D^2 = \Omega(M^2)$ depends on the shape of the noise. When $\sigma_k$ is increasing in the first half and decreasing in the second half, as in Example~\ref{example2}, the total variation is bounded by $D^2 \le 2M^2$. More generally, if the noise can be decomposed into $K$ piece-wise monotone fragments, then the bound $D^2 \le KM^2$ holds.

\begin{algorithm}[tbp]
	\caption{Moment Estimation SGD ($x_{1}, T,  c, m $)}\label{algo:sigma}	    
	\begin{algorithmic}[1]
		\STATE Initialize $\hat{\sigma}_1 = \frac{\|g_{1,1} - g_{1,2}\|^2}{2}$, where $g_{1,1}, g_{1,2}$ are two independent stochastic gradients at $x_1$. 
		\FOR{$t = 1, 2, ..., T / 2$}
		\STATE Query two  stochastic gradients $g_{t,1}, g_{t,2}$ at $x_t$.
		\STATE Update 
		\begin{align*}
			&  \overline{g_t} = \frac{g_{t,1}+g_{t,2}}{2} \text{ and } \eta_{t}= \tfrac{c}{\hat{\sigma}_t + m}  \\
			&x_{t+1} = x_{t} - \eta_{t} \overline{g_t} \quad  \\
			&\hat{\sigma}_{t+1}^2 = \beta \hat{\sigma}_{t}^2 + (1-\beta) \frac{\|g_{t,1} - g_{t,2}\|^2}{2}
		\end{align*} 
		
		\ENDFOR
		\STATE \textbf{Return} $x_I$ where $I$ is the random variable such that $\mathbb{P}(I = i) \propto \eta_i$.
	\end{algorithmic}
	
\end{algorithm}

\subsection{Design and analysis of our proposed  algorithm}

\begin{table*}[ht]
	\centering 
	\ra{1}
	\resizebox{0.7\textwidth}{!}{\begin{tabular}{@{}lcrcrcrcrc@{}}\toprule
			& {Worst} & \phantom{a} & {Adaptive} &  \phantom{a} & {Dynamic} & \phantom{a} & {Moment Est.} &
			\phantom{a} \\ \midrule
			{$\alpha = 1/9$} 
			& $T^{-\frac{1}{2}}$  && $T^{- 5 / 9}$ && $T^{- 11   / 18  }$  &&$T^{- 11 / 18}$   \\
			{$\alpha = 1/8$} 
			& $T^{-\frac{1}{2}}$  && $T^{- 9 / 16}$ && $T^{- 11   / 18  }$  &&$T^{- 5 / 8}$   \\
			\bottomrule
	\end{tabular}}
	\caption{Comparison of convergence rate under the mountain shape noise $\sigma_k$, where constant terms are omitted. The parameter $\alpha  = 1/9$ is the exponent controlling the ratio between $\max \sigma_k$ and $\min \sigma_k$, indicating the significance in noise changes. }\label{tab:example_rate}
\end{table*}

With the bounded total variation in hand, we are now ready to exploit the continuity of $\sigma_k$. Without loss of generality, we use the empirical variance estimator \eqref{eq:ub-var-estimator} with two samples per iteration. We construct an online estimator using an exponential moving average:
\begin{align}\label{eq:exp-est}
	\hat{\sigma}_{t+1}^2 = \beta \hat{\sigma}_{t}^2 + (1-\beta) \frac{\|g_{t,1} - g_{t,2}\|^2}{2}, \tag{ExpMvAvg}
\end{align}
where $g_{t,1}$ and $g_{t,2}$ are two samples of the gradient at $t$-th iteration. Note that \emph{$t$ is the index for algorithm iteration, whereas $k$ is the index for number of oracle calls}. The oracle call at iteration $t$ corresponding to the $(2t)_{th}$ and $(2t+1)_{th}$ call to the stochastic oracle. 

The moving average is very similar to the one used in the modern adaptive methods such as RMSProp~\citep{Tieleman2012}, Adam~\citep{Kingma2014}, etc. One major difference compared to the current adaptive methods is that we are estimating the variance of gradients, whereas RMSProp/Adam directly accumulate the gradient square in a coordinate-wise manner.

The above estimator combined with dynamic stepsizes lead to our design of Algorithm~\ref{algo:sigma}. To avoid the explosion of stepsize when $\hat{\sigma}_t$ is underestimated, we include a constant correction term $m$ in the denominator of $\eta_t$. Such correction term is also~commonly used in the practical implementation of adaptive methods~\cite{Duchi2011, Kingma2014}. In reinforcement learning, this term is sometimes referred to as the  exploration bonus~\citep{Strehl2008, Azar2017}. 

Algorithm~\ref{algo:sigma} gives us the following convergence guarantee.
\begin{theorem}[Main result] \label{thm:convex}
	Under  Assumptions~\ref{assump:fourth moment},\ref{assump:sum}, 
	with probability at least $1/2$, the iterates generated by Algorithm~\ref{algo:sigma} using parameters $\beta = 1 - 2T^{-2/3}$, $m = 4\sqrt{D^2+M^2} T^{-\frac{1}{9}} \ln(T)^{\frac{1}{2}}$, $c= \tfrac{R}{\sqrt{T}}$ satisfy
	\vspace{-0.1cm}
	\[f(\overline{x}_T) -f^* \le \tfrac{64 R}{\sqrt{T}} \cdot \rbr*{\tfrac{1}{T}\tsum_{k=1}^T \frac{1}{\sigma_k+m}}^{-1} := \epsilon_{ours}.\]
\end{theorem}

\begin{corollary}
	Our result directly implies a $1 - \delta $ high probability convergence rate, by restarting it $2\log(1/\delta)$ times. An additional $\log(1/\delta)$ dependency will be introduced in the complexity, as in standard high probability results \cite{Nemirovski2009, Jin2017b, Fang2018a}.
\end{corollary}

The main challenge in proving the theorem is to effectively bound the estimation error $|\hat{\sigma}^2_t - \sigma^2_{2t}/2 - \sigma^2_{2t+1} / 2|$. As $\hat{\sigma}^2_t$ is an exponential average, the past errors accumulates into the current estimator, which requires a careful bound using concentration and the total variation of $\sigma_k$. In particular, the decay parameter~$\beta$ plays a critical role, determining the contribution of past gradients in the current estimator. A consequence of using past gradients in the estimate is that the online estimator $\hat{\sigma}_k$ is no longer unbiased. The proposed choice of $\beta$ and $m$ carefully balances the bias error and the variance error, leading to a sublinear regret, see~Appendix~\ref{appendix:keylemma}. 


\paragraph{Understanding the convergence rate} Due to the correction constant $m$, the obtained convergence rate inversely depends on $\sum_{k=1}^T \frac{1}{\sigma_k+m}$, instead of the $\sum_{k=1}^T \frac{1}{\sigma_k}$ dependency in the dynamic bound. This additional term makes the comparison less straightforward, especially when some of the $\sigma_k$ are small. We provide several scenarios to facilitate the comparison. 

\begin{corollary}\label{cor:ideal baseline}
	If the ratio $M/(\min_k \sigma_k) \le  T^{\frac{1}{9}}$, then the Moment Estimation SGD method converges in the same order as the dynamic error bound~$\epsilon_{dynamic}$.
\end{corollary}

This result is remarkable since our proposed method does not require any knowledge of $\sigma_k$ values, and yet it achieves the dynamic rate. In other words, the exponential moving average estimator successfully adapts to the variation in~noise, allowing faster convergence than adaptive/worst bounds.  In particular when taking $\alpha = 1/9$ in Example~\ref{example2}, we get the following error bounds in Table~\ref{tab:example_rate}. The comparison between different error bounds suggest that moment estimation can non-trivially improve the convergence rates achieved by conventional analysis for adaptive step sizes.

\begin{corollary}\label{cor:std baseline}
Let $\sigma_{avg}^2 = \tsum \sigma_k^2/T$ be the average second moment. If $M/ \sigma_{avg}\le~T^{\frac{1}{9}}$, then adaptive method is no slower than  the adaptive error bound \eqref{eq:adaptive}. 
\end{corollary}

The condition in Corollary~\ref{cor:std baseline} is strictly weaker than the condition in  Corollary~\ref{cor:ideal baseline}, which means even though an adaptive method may not match the dynamic bound, it can still be non-trivially better than the adaptive bound. This case happens for instance when $\alpha > \frac{1}{9}$ in Table~\ref{tab:example_rate}, where the proposed method is $\cO(T^{\frac{1}{9}})$ faster than the adaptive bound. Indeed,  $\cO(T^{\frac{1}{9}})$ is the maximum improvement one can expect according to our current analysis. 

We now proceed to evaluate different algorithms and bounds  with both synthetic and deep learning experiments.

\begin{figure*}[htb]
	\centering
	
	\begin{minipage}{.3\linewidth}
		\includegraphics[width=\linewidth]{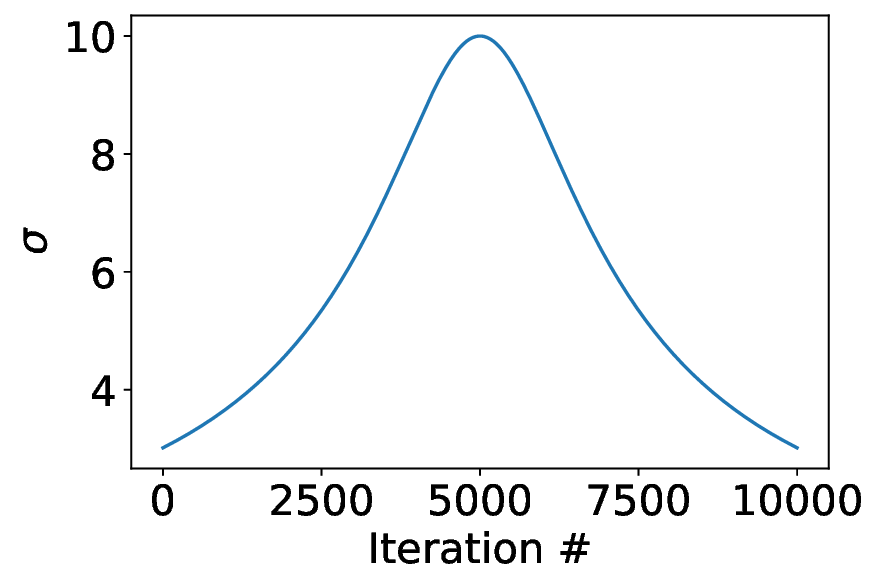}
	\end{minipage}
	\begin{minipage}{.33\linewidth}
		\includegraphics[width=\linewidth]{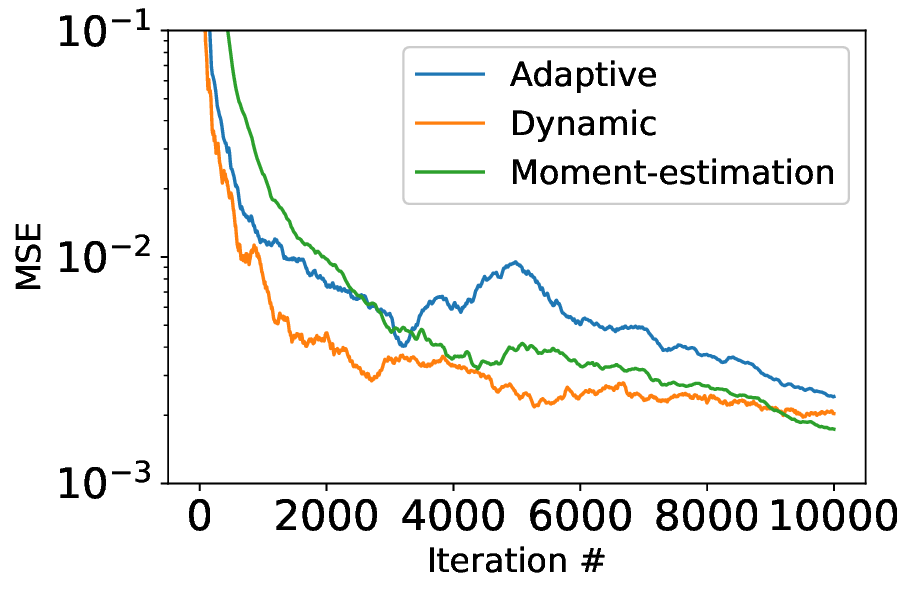}
	\end{minipage}
	\begin{minipage}{.3\linewidth}
		\includegraphics[width=\linewidth]{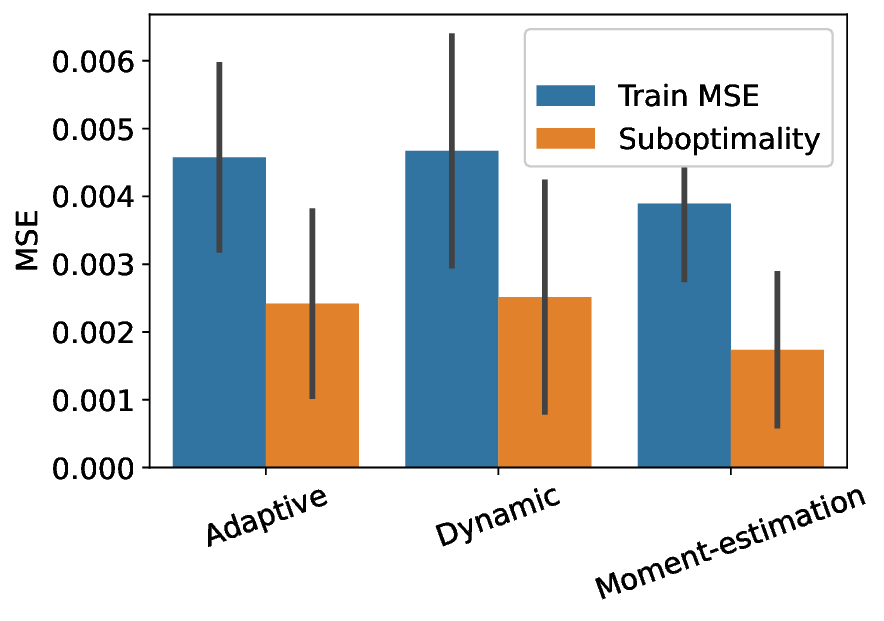}
	\end{minipage}
	\caption{Results for synthetic least square problems.  The three plots from left to right corresponds to the shape of noise, the average MSE out of 10 random runs and the numerical values of suboptimality.}
	\label{fig:synth}
\end{figure*}

\begin{figure*}[htb]
	\centering
	\begin{minipage}{.24\linewidth}
		\includegraphics[width=\linewidth]{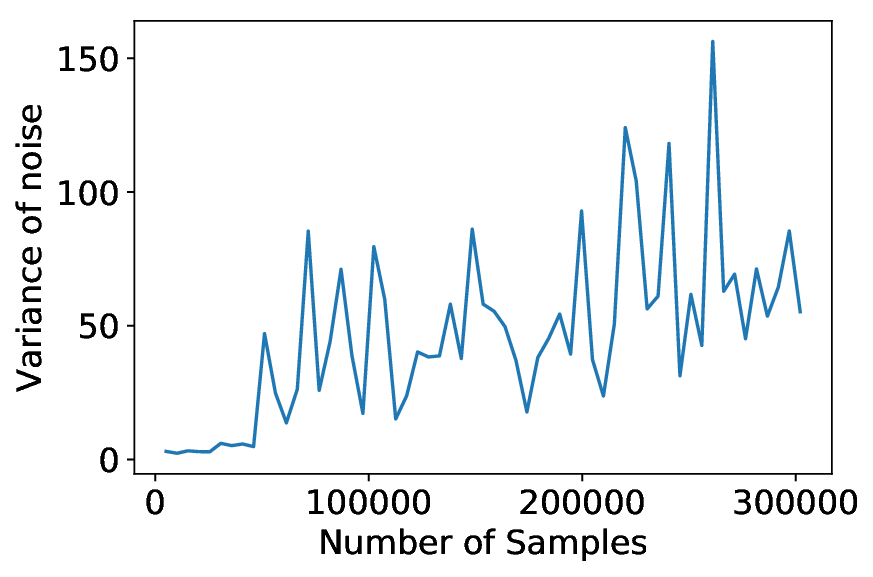}
	\end{minipage}
	\begin{minipage}{.24\linewidth}
		\includegraphics[width=\linewidth]{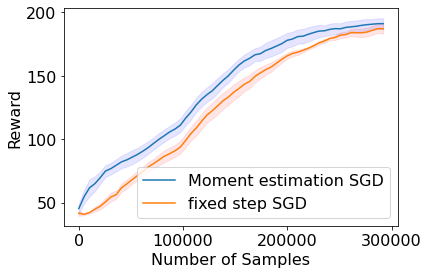}
	\end{minipage}
	\begin{minipage}{.24\linewidth}
		\includegraphics[width=\linewidth]{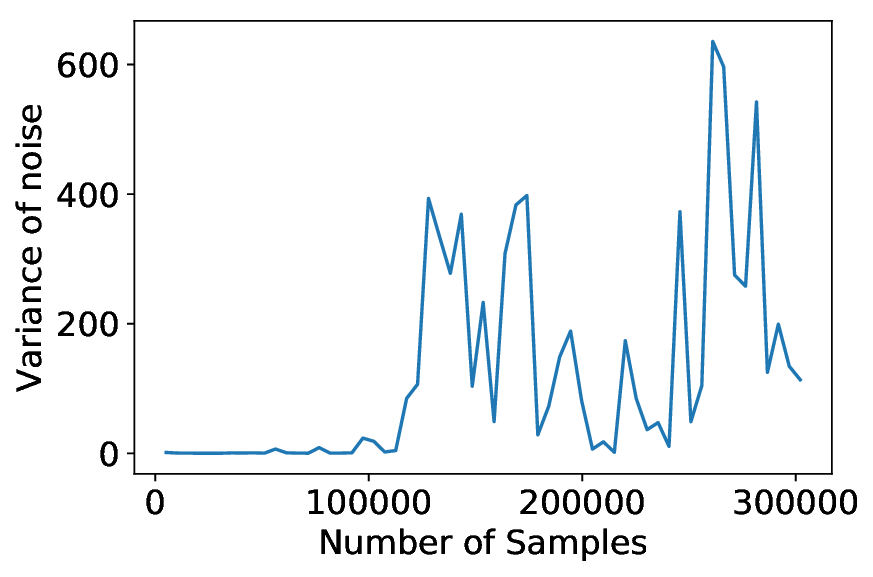}
	\end{minipage}
	\begin{minipage}{.24\linewidth}
		\includegraphics[width=\linewidth]{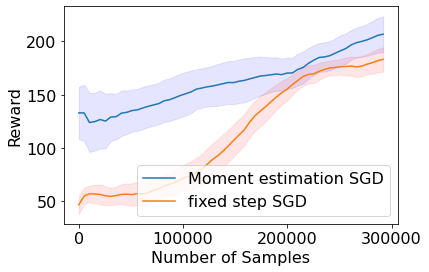}
	\end{minipage}
	\caption{The left two figures plot the noise level and the average reward during training a controller for the walker task. The right two plots correspond to the cartpole swing-up task.}
	\label{fig:rl}
\end{figure*}

\section{Experiments}

In this section, we present three sets of experiments: synthetic least squares, policy optimization for mujoco tasks and neural network training on Cifar10 dataset. Our synthetic experiments verify the theory prediction, whereas our deep learning experiments show that the studied moment-estimation algorithm is practical and performs as well as common baselines.

\begin{figure*}[htb!]
	\centering
	\begin{minipage}{.27\linewidth}
		\includegraphics[width=\linewidth]{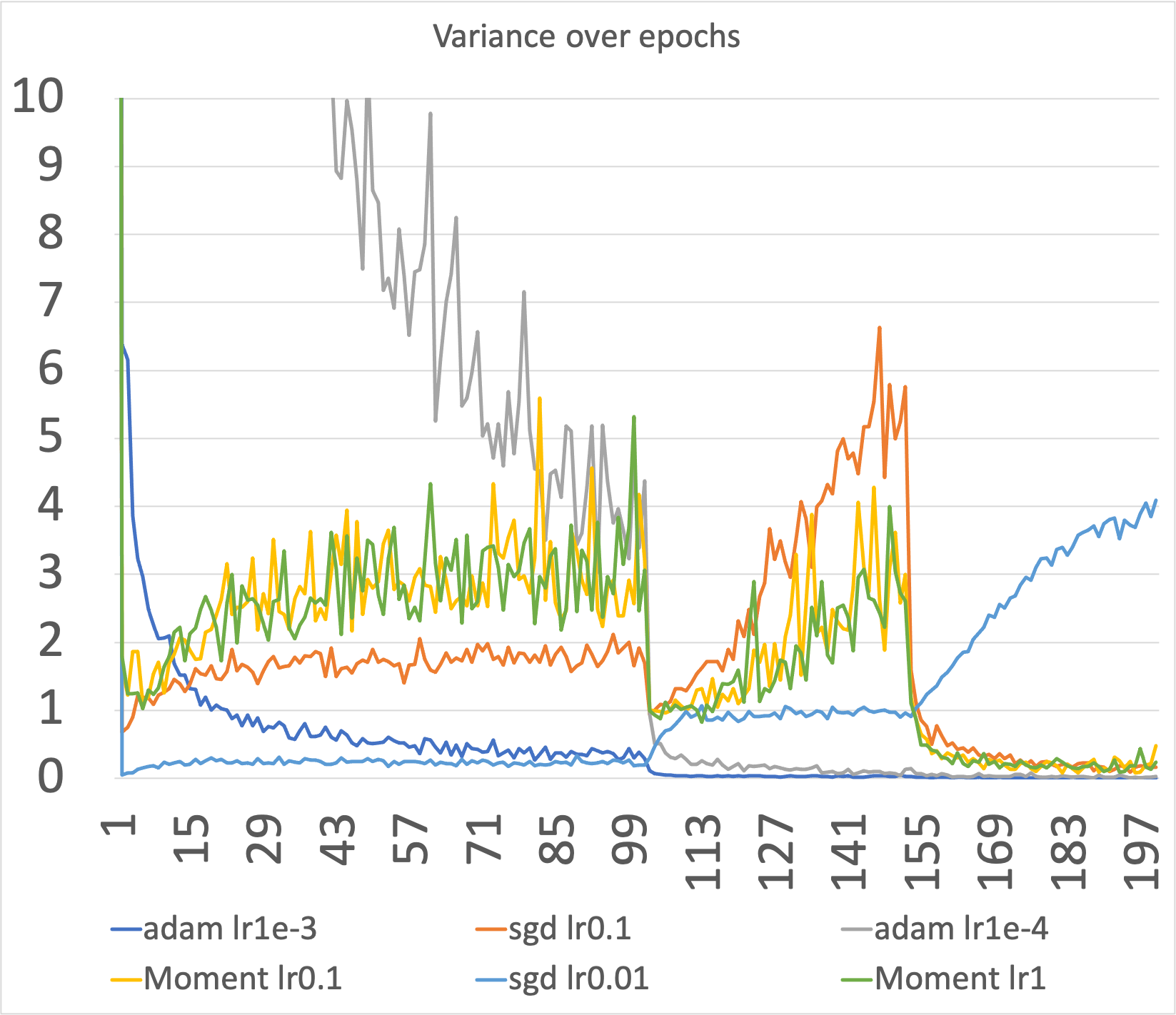}
	\end{minipage}
	\begin{minipage}{.27\linewidth}
		\includegraphics[width=\linewidth]{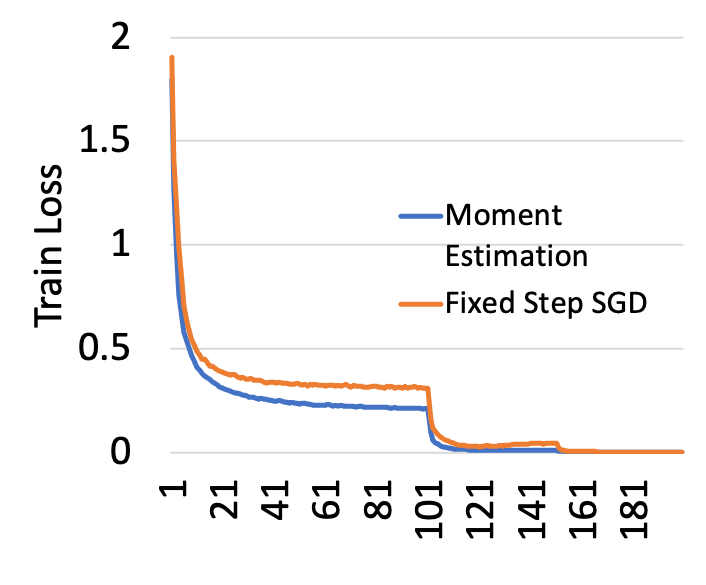}
	\end{minipage}
	\begin{minipage}{.29\linewidth}
		\includegraphics[width=\linewidth]{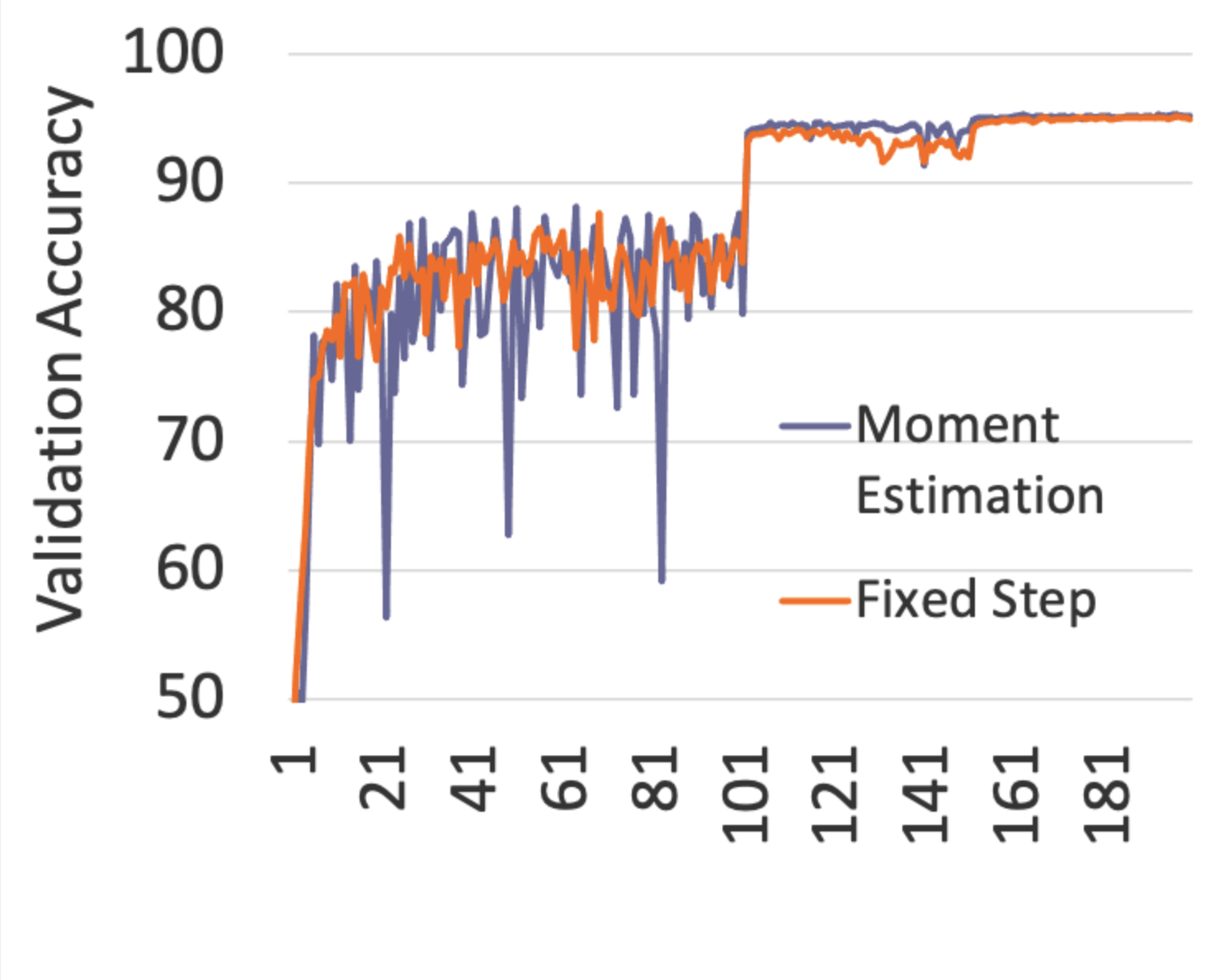}
	\end{minipage}
	\caption{The left figure plots the noise level while training a ResNet model on Cifar10 for 200 epochs. It shows that noise levels are very much algorithm dependent. The second left figure shows that the variance in stochastic gradient almost equals the variance. The right two figures show that in Cifar10 training, the moment estimation algorithm performs as well as the fine-tuned SGD baseline.}\label{fig:cifar}
\end{figure*}

\subsection{Synthetic experiments}
In the synthetic experiment, we generate a random linear regression dataset using the {\tt sklearn} library. We design the stochastic oracle as full gradient with injected Gaussian noise, whose coordinate-wise standard deviation $\sigma$ is shown in the top row figures of Fig~\ref{fig:synth}. We then run the three algorithms discussed in this work: the fixed best step-size, the dynamic step size and the moment estimation algorithm in Alg~\ref{algo:sigma}. We fine-tune the step sizes for each algorithm by grid-searching. We repeat the experiment for 10 runs and show the average training trajectory in the second row in Fig~\ref{fig:synth}.  From Figure~\ref{fig:synth}, we see that moment estimation algorithm can achieve comparable performance against the dynamic version and outperforms the fixed step size SGD.

\subsection{Policy optimization}
In this set of experiments, we simulated the proximal policy optimization algorithm in the nonstationary environment described in \cite{DulacArnold2019}. We consider two Mujoco tasks: one in which a walker gets reward for walking; another in which a cart-pole gets reward for swinging up the pole. The original implementation\footnote{\url{https://github.com/google-research/realworldrl_suite}} update the parameters with ADAM~\cite{Kingma2014} and gradient clipping, which as we will show later, closely connects to noise-dependent step size. To highlight the comparison between fixed step size SA and adaptive step size SA, we replaced the update with the constant baseline and the moment estimation algorithm. 

During training, we estimate the variance level by sampling a batch of 5 thousand state, action pairs. For each algorithm, we fine-tune the step sizes by grid-searching among $10^k,$ where $k$ is an integer. The estimated variance of noise and average reward over 5 runs are plotted in Figure~\ref{fig:rl}. The result shows that the noise level is indeed nonstationary during training, where the SA algorithm can benefit from an adaptive step size. We note that the adaptive step size method learns much faster than fixed stepsize in the cartpole task over the first few iterations. The advantage might result from the low noise level in initial epochs of the policy gradient due to the simplicity of swinging up a pole.

\subsection{Neural network training}
We will now discuss the application of our proposed method to neural network training. In order to apply our result in such settings, the challenge  is not merley about  extending the results to nonsmooth scenarios or non-convex cases. Rather, the main challenge we face is that noise levels in stochastic optimization for  neural network training are step-size dependent. The source of nonstationary noise in this setting is entirely \emph{endogenous}, i.e., it is determined by the iterative output $x_t$. In such settings, it is unclear how baselines could be defined, or improvement could be quantified. 

To this end, we show in Figure~\ref{fig:cifar} that the noise trajectories are different for training ResNet18 on Cifar10 with different algorithms and hyperparameters. Even though our theoretical guarantees do not directly apply to this setting, we can apply our adaptive step-size algorithm and still exploit the variations in noise. 

Directly applying Algorithm~\ref{algo:sigma} may be cumbersome on standard image classification pipelines due to requiring an unbiased within-batch variance estimator. Instead, we notice from Figure~\ref{fig:cifar} that the noise level dominates the gradient norm in neural network training, and hence we can simply use the empirical second moment $\E[\|g_k\|^2]$ as a substitute for the empirical variance $\E[\|g_k - \nabla f(x_k) \|^2]$. This gives us the following update,
\begin{align*}
	x_{k+1} = x_{k} - \eta_{k} g_k \quad \text{ with } \quad \eta_{t}= \tfrac{c}{\hat{m}_k + m}, \\
\text{ and }	\hat{m}_{k+1}^2 = \beta \hat{m}_{k}^2 + (1-\beta) \frac{\|g_{k}\|^2}{2}.
\end{align*} 
We should point out that interestingly, the above update is {\it exactly the same} as in the RMSProp algorithm~\cite{Tieleman2012} if the step sizes were coordinate-wise. Since the moment estimation algorithm is very similar to popular optimizers for training Cifar10, we do not expect it to significantly outperform the well tuned baselines. Instead, we highlight that our analysis provides theoretical evidence for the popularity of moment estimation techniques in practice.

Such observations point to the following interesting question: ``\emph{why is fixed stepsize SGD minimax optimal~\citep{Arjevani2019}, yet adaptive methods such as RMSProp and ADAM outperforms fixed step size SGD in many real world settings?}" Alongside with many recent works~\citep{Reddi2019,Wilson2017,Zhang2019, Ward2018, luo2019adaptive, liu2020adam}, we believe that our more fine-grained analysis provides a new perspective and motivates new avenues for proving the effectiveness of adaptive algorithms such as ADAM and RMSProp.

\section{Conclusions and Discussions}\label{sec: discussion}

In this work, we provided a new perspective for instance-dependent complexity of stochastic approximation methods. We first categorized existing instance-dependent error bounds into different levels based on dominance relations. We then proposed a new dynamic error bound that dominates known ones. Simple algorithms that achieves this bound requires knowing the exact noise levels and is not implementable. To address this issue, we showed that when noise levels have bounded total variation,  moment estimation can achieve the desired rate. Our results are validated by both synthetic and real-world experiments. We believe the instance complexity we developed shed new insights to the following interesting question: ``\emph{why is fixed stepsize SGD minimax optimal~\citep{Arjevani2019}, yet adaptive methods such as RMSProp and ADAM outperforms fixed step size SGD in many real world settings?}"

Many important instance-complexity problems are still open. First, in traditional complexity theory, instance dependent lower bounds can sometimes be tight up to constants~\cite{afshani2017instance}. However, determining the lower bounds for stochastic approximation instances requires a reformulation of the complexity definition, such as what information is available. 
For example, in the dynamic bounds setup, a different step size choice,
\begin{align*}
    \eta_k = R /  (\sigma_k^2 \sqrt{\textstyle\sum_{i=1}^T \tfrac{1}{\sigma_i^2}} ),
\end{align*}
may lead to a better error bound. Yet, this kind of step size is at the same time dynamic and non-causal (depends on information from future iterates). Therefore, information dependence needs to be properly integrated in the lower bound of instance complexity.  

Second, beyond settings such as smooth-convex problems, we currently do not know of any faster instance-dependent bounds. 
Even if better bounds can be achieved when noise levels are known, it could still be unclear whether it can be achieved by practical and implementable algorithms. Another question is whether in our setup where the function is smooth and convex the $T^{1/9}$ factor can be improved.  We believe understanding  these problems can provide more insight into practical algorithm performances and lead to invention of new gradient based algorithms.

Last but not least, an apparent limitation of the stochastic approximation setting is the assumption of an exogenous  noise. It is usually not satisfied in the  standard empirical risk minimization framework, where the noise is iterate dependent, i.e. $x$-dependent. Note that the iterates generated by one algorithm is mostly very different from the iterates generated by another one, how to appropriately quantify the state-dependency such that we can derive non-vacuous instance complexity result becomes challenging. Though simplified, we believe our work provides a solid first step towards this ultimate goal. 

\subsection*{Acknowledgments} SS acknowledges support from an NSF CAREER award (number 1846088), and NSF CCF-2112665 (TILOS AI Research Institute). JZ acknowledges support from a IIIS young scholar fellowship.

\bibliography{main}
\bibliographystyle{icml2022}

\newpage
\appendix
\onecolumn

\section{Proof of Theorem~\ref{thm:baselines}}

\begin{proof} 
Using the fact that $\frac{1}{L} \| \nabla f(y) - \nabla f(x)  \|^2 \le \langle \nabla f(y) - \nabla f(x), y-x \rangle $, we have
\begin{align*}
    \E[\| x_{k+1} -x^* \|^2 | \F_k]  & =  \| x_{k} -x^* \|^2 - 2\eta_k \langle \nabla f(x_k), x_k - x^* \rangle + \eta_k^2 \| \nabla f(x_k)^2\| + \eta_k^2 \sigma_k^2 \\
    & \le  \| x_{k} -x^* \|^2 - \eta_k(2-L\eta_k) \langle \nabla f(x_k), x_k - x^* \rangle  + \eta_k^2 \sigma_k^2 \\
    & \le  \| x_{k} -x^* \|^2 - \eta_k \langle \nabla f(x_k), x_k - x^* \rangle  + \eta_k^2 \sigma_k^2
\end{align*}
As $f(x_k)-f^* \le \langle \nabla f(x_k), x_k - x^* \rangle$. We have 
\[ \eta_k (f(x_k)-f^*) \le \| x_{k} -x^* \|^2 - \E[\| x_{k+1} -x^* \|^2 | \F_k] + \eta_k^2 \sigma_k^2\]
Taking expectation and telescoping yields 
\[ \E \left [ \sum_{k=1}^T \eta_k (f(x_k)-f^*) \right ] \le \| x_{1} -x^* \|^2 - \E[\| x_{T+1} -x^* \|^2 ] + \sum_{k=1}^T \eta_k^2 \sigma_k^2\]
\end{proof}

\section{Proof of Lemma ~\ref{lemma:ordering}} 

The only nontrivial one is $\epsilon_{\text{dynamic}} \preceq \epsilon_{\text{adaptive}}$. This follows from Jensen's inequality $ \E[X]^{-2} \le \E[X^{-2}] $,  and we have 
\[ (\tfrac{1}{T}\tsum_k \tfrac{1}{\sigma_k})^{-2} \le \tfrac{1}{T}\tsum_k \sigma_k^2,\].

\section{Summation series in Example~\ref{example2}}

In Example~\ref{example2}, we have
\begin{align*}
\sigma_k = \frac{1}{\sqrt{1+ T^{\alpha} \left (\frac{2k}{T} - 1 \right )^2}}
\end{align*}
\begin{itemize}
    \item The second moment of $\sigma_k$ is given by 
    \begin{align*}
        \frac{1}{T} \sum_{k=1}^T \sigma_k^2 = \frac{1}{T} \sum_{k=1}^T \frac{1}{{1+ T^{\alpha} \left (\frac{2k}{T} - 1 \right )^2}} & \quad \sim \quad  \frac{1}{T} \int_{0}^{T} \frac{1}{{1+ T^{\alpha} \left (\frac{2x}{T} - 1 \right )^2}} dx \\
        & \stackrel{u := \frac{2x}{T}-1}{=}  \frac{1}{T}\int_{-1}^{1} \frac{1}{1+ T^{\alpha} u^2 } \frac{du}{\frac{2}{T}} \\
        & = \int_{0}^{1} \frac{1}{1+ T^{\alpha} u^2 } du \\
        & = \left [\frac{1}{T^{\frac{\alpha}{2}}} arctan (x) \right ]_{0}^{T^{\frac{\alpha}{2}}} \sim \frac{1}{T^{\frac{\alpha}{2}}}
    \end{align*}
    This implies that the adaptive bound \ref{eq:adaptive} is of order $O(T^{-\left (\frac{1}{2}+\frac{\alpha}{4} \right )} )$.
    \item The harmonic sum of $\sigma_k$ is given by
    \begin{align*}
        \frac{1}{T} \sum_{k=1}^T \frac{1}{\sigma_k} = \frac{1}{T} \sum_{k=1}^T {\sqrt{1+ T^{\alpha} \left (\frac{2k}{T} - 1 \right )^2}} \,\,  &  =    O\left ( \frac{1}{T} \sum_{k=1}^T 1 + T^{\frac{\alpha}{2}} \left |\frac{2k}{T}-1 \right |  \right ) \\
        & = O\left (1 + \frac{1}{T} T^{\frac{\alpha}{2}} \frac{2 \sum_{k= 0}^{\frac{T}{2}} (T-2k)}{T} \right )\\
        & = O\left (1 + \frac{1}{T} T^{\frac{\alpha}{2}} \frac{T^2}{T} \right ) = O\left (T^{\frac{\alpha}{2}} \right )
    \end{align*}
    This implies that the dynamic bound\ref{eq:dynamic} is of order $O(T^{-\left (\frac{1}{2}+\frac{\alpha}{2} \right )} )$.

\end{itemize}

\section{Key Lemma: Total estimation error of the variance estimator} \label{appendix:keylemma}
\begin{lemma}\label{thm:online}
	Under Assumptions~\ref{assump:sum},
	taking $\beta = 1 - 2T^{-2/3}$, the total estimation error of the $\hat{\sigma}_k^2$ based on~\eqref{eq:exp-est} is bounded by:
	\begin{align*}
	\E \left [\sum_{t=1}^{T/2} | \hat{\sigma}_{t}^2 - \sigma_{2t}^2 / 2 - \sigma_{2t+1}^2 / 2 | \right ] \le 2(D^2 + M^2 ) T^{2/3} \ln (T^{2/3})
	\end{align*}
\end{lemma}

\begin{proof} 
	On a high level, we decouple the error in a bias term and a variance term. We use the total variation assumption to bound the bias term, and use the exponential moving average to reduce variance. Then we pick $\beta$ to balance the two terms.
	
	For simplicity, we denote $ \text{var}_t^2 =  \sigma_{2t}^2 / 2 + \sigma_{2t+1}^2 / 2$.
	
	From triangle inequality, we have 
	\begin{equation}
	\sum_{t=0}^{T/2}   \E \left [| \hat{\sigma}_{t}^2 - \text{var}_t^2 | \right ] \le  \sum_{t=1}^{T/2}  \underbrace{ \E \left [| \hat{\sigma}_{t}^2 -  \E [\hat{\sigma}_{t}^2] |\right ]}_{\text{Variance term}} + \sum_{t=1}^{T/2} \underbrace{ \left | \E [\hat{\sigma}_{t}^2]  - \text{var}_t^2  \right | }_{\text{Bias term}}
	\end{equation}
	We first bound the bias term. By definition of $\hat{\sigma}_t$, we have 
	\begin{align*}
	\E [\hat{\sigma}_{t}^2] - \text{var}_t^2 & = \beta \E [ \hat{\sigma}_{t-1}^2] + (1-\beta)\text{var}_{t-1}^2 - \text{var}_t^2 \\
	& = \beta (\E [ \hat{\sigma}_{t-1}^2] - \text{var}_{t-1}^2) + (\text{var}_{t-1}^2- {\text{var}}_{t}^2) 
	\end{align*}
	Hence by recursion,
	\begin{align*}
	\E [\hat{\sigma}_{t}^2] - \text{var}_t^2 & = \beta^{t-1} \underbrace{(\E [ \hat{\sigma}_{1}^2] - {\text{var}}_{1}^2)}_{=0} + \beta^{t-2} ( {\sigma}_{1}^2- \text{var}_2^2) + \cdots +  (\text{var}_{t-1}^2- {\text{var}}_{t}^2) 
	\end{align*}
	Therefore, the bias term could be bounded by
	\begin{align*}
	\sum_{t=1}^{T/2} \left | \E [\hat{\sigma}_{t}^2]  - \text{var}_t^2  \right | & \le \sum_{t=1}^{T/2} \sum_{j=1}^{t-1} \beta^{t-1-j}  \left |{\text{var}}_{j}^2- {\text{var}}_{j+1}^2 \right | \\
	& \le \sum_{k=1}^T \sum_{j=1}^{k-1} \beta^{k-1-j}  \left |{\sigma}_{k}^2- {\sigma}_{k+1}^2  \right | \\
	& = \sum_{k=1}^{T-1} \left |{\sigma}_{k}^2- {\sigma}_{k+1}^2 \right | \sum_{j=0}^{T-1-k} \beta^{j}   \\
	& \le \frac{1}{1-\beta}\sum_{k=1}^{T-1} \left |{\sigma}_{k}^2- {\sigma}_{k+1}^2 \right | \\
	& \le \frac{D^2}{1-\beta} \quad \text{(From Assumption (\ref{assump:sum}))}
	\end{align*}
	The first inequality follows by traingle inequality. The third inequality uses the geometric sum over $\beta$. To bound the variance term, we remark that 
	\[ \hat{\sigma}_{t}^2 = (1-\beta)y_{t-1}^2 + (1-\beta) \beta y_{t-2}^2 + \cdots + (1-\beta)\beta^{t-2} y_1^2 + \beta^{t-1} y_0^2.   \]
	where we denote $$y_t = \frac{\|g_{t,1} - g_{t,2}\|^2}{2}.$$
	Hence from independence of the gradients, we have 
	\begin{align*}
	\E \left [| \hat{\sigma}_{t}^2 -  \E [\hat{\sigma}_{t}^2] |\right ] & \le \sqrt{\text{Var}[\hat{\sigma}_{t}^2]} \\
	& = \sqrt{ \text{Var}[(1-\beta) y_{t-1}^2] +  \text{Var}[ (1-\beta)\beta y_{t-2}^2] + \cdots +  \text{Var}[(1-\beta)\beta^{t-2} y_1^2] +  \text{Var}[\beta^{t-1} y_0^2] }\\
	& \le  \sqrt{ (1-\beta)^2 +  (1-\beta)^2\beta^2  + \cdots +  (1-\beta)^2\beta^{2(t-2)} +  \beta^{2(t-1)} } M^2 ,
	\end{align*}
	where $M^2$ is an upperbound on the variance. The first inequality follows by Jensen's inequality. The second equality uses independence of $y_i$ given $g_1, ..., g_{i-1}$. The last inequality follows by assumption~\ref{assump:sum}. 
	
	We distinguish two cases, when $t$ is small, we simply bound the coefficient by 1, i.e.
	\[  \sqrt{ (1-\beta)^2 +  (1-\beta)^2\beta^2  + \cdots +  (1-\beta)^2\beta^{2(t-2)} +  \beta^{2(t-1)} } \le 1 \]
	When $t$ is large such that $t \ge  1 + \gamma$, with $\gamma =  \frac{1}{2(1-\beta)} \ln(\frac{1}{1-\beta})$, we have $\beta^{2(t-1)} \le 1- \beta $, thus
	\begin{align*}
	&  \sqrt{ (1-\beta)^2 +  (1-\beta)^2\beta^2  + \cdots +  (1-\beta)^2\beta^{2(t-2)} +  \beta^{2(t-1)} }  \\
	\le & \sqrt{ \frac{(1-\beta)^2}{1- \beta^2} + \beta^{2(t-1)}  }  \\
	\le & \sqrt{ \frac{(1-\beta)^2}{1- \beta^2} + (1-\beta)  } \\
	\le & \sqrt{ 2(1-\beta)  } 
	\end{align*}
	The second inequality follows by $t \ge  1 + \gamma$, with $\gamma =  \frac{1}{2(1-\beta)} \ln(\frac{1}{1-\beta})$. Therefore, when $t \ge  1 + \gamma$,
	\begin{align*}
	\E \left [| \hat{\sigma}_{t}^2 -  \E [\hat{\sigma}_{t}^2] |\right ] \le \sqrt{ 2(1-\beta)  } M
	\end{align*}

	Therefore, substitute in the above equation into the 
	\begin{align*}
	\sum_{t=1}^{T/2} \E \left [| \hat{\sigma}_{t}^2 -  \E [\hat{\sigma}_{t}^2] |\right ] & = \sum_{t=1}^{\gamma}  \E \left [| \hat{\sigma}_{t}^2 -  \E [\hat{\sigma}_{t}^2] |\right ] + \sum_{t= \gamma+1}^{T/2} \E \left [| \hat{\sigma}_{t}^2 -  \E [\hat{\sigma}_{t}^2] |\right ] \\
	& \le (\gamma + (T-\gamma)\sqrt{ 2(1-\beta)  } ) M^2  \\
	\end{align*} 
	Summing up the variance term and the bias term yields,
	\begin{equation}
	\sum_{t=0}^{T/2}   \E \left [| \hat{\sigma}_{t}^2 - \text{var}_t^2 | \right ] \le  \frac{D^2}{1-\beta} + (\gamma + (T-\gamma)\sqrt{ 2(1-\beta)  } ) M^2
	\end{equation}
	Taking $\beta = 1 - T^{-2/3}/2$ yields, 
	\begin{equation}
	\sum_{t=0}^{T/2}   \E \left [| \hat{\sigma}_{t}^2 - \text{var}_t^2 | \right ] \le 2(D^2 + M^2 ) T^{2/3} \ln (T^{2/3})
	\end{equation}
	

\end{proof}

\section{Proof of Theorem~\ref{thm:convex}}

On a high level, the difference between the adaptive stepsize and the idealized harmonic stepsize mainly depends on the estimation error $|\hat{\sigma}_k^2 -\sigma_k^2|$, which has a sublinear regret according to Lemma~\ref{appendix:keylemma}. Then we carefully integrate this regret bound to control the derivation from the idealized algorithm, reaching the conclusion.

\begin{proof}
	By the update rule of $x_{t+1}$, we have,
	\begin{align*}
	\|x_{t+1} -  x^*\|^2 = \|x_t - \eta_t \overline{g_t} -  x^*\|^2 = \|x_{t} -  x^*\|^2 - 2\eta_t \iprod{\overline{g_t}}{x_t - x^*} + \eta_t^2 \|\overline{g_t}\|^2.
	\end{align*}
	Noting that the stepsize $\eta_t$ is independent of $\overline{g_t}$ conditioned on $\overline{g_{t-1}}$. Recall that for simplicity, we denote $ \text{var}_t^2 =  \sigma_{2t}^2 / 2 + \sigma_{2t+1}^2 / 2$.
Taking  expectation with respect to $g_{t}$ conditional on the past iterates lead to
	\begin{align*}
	2\eta_t (f(x_t) - f^*) & \le 2\eta_t \iprod{\nabla f(x_t)}{x_t - x^*}  \\
	& =  \E [ 2\eta_t \iprod{\overline{g_t}}{x_t - x^*} | x_t, \cdots ,x_1] \\
	& =  - \E[\|x_{t+1} -  x^*\|^2 | x_t, \cdots ,x_1] + \|x_t -  x^*\|^2 + \eta_t^2 (\|\nabla f(x_t)\|^2 + \text{var}_t^2)\\
	& \le  - \E[\|x_{t+1} -  x^*\|^2 | x_t, \cdots ,x_1] + \|x_t -  x^*\|^2 +  \eta_t^2\sigma_t^2 + L\eta_t^2 \iprod{\nabla f(x_t)}{x_t - x^*}.
	\end{align*}
	Recall that $R = \| x_1-x^*\|$, taking expectation and sum over iterations $k$, we get
	\begin{align*}
	\E[(\tsum_{t=1}^{T/2} \eta_t) (f(\overline{x}_T) - f^*)] \le R^2 +\E[ \sum_{t=1}^{T/2} \eta_t^2  \text{var}_t^2].
	\end{align*}
	Hence by Markov's inequality,  with probability at least $3/4$, 
	\begin{equation}\label{eq:high proba}
	(\tsum_{t=1}^{T/2} \eta_t) (f(\overline{x}_T) -f^*)  \le 4 \E[2 (\sum_{t=1}^{T/2} \eta_t) (f(\overline{x}_T) -f^*)] \le  4 (R^2 + \E[ \sum_{t=1}^{T/2} \eta_t^2 \text{var}_t^2] ) .
	\end{equation}
	Now we can upper bound the right hand side, indeed 
	\begin{align}\label{eq:upper bound right}
	\sum_{t=1}^{T/2} \E[\eta_t^2 \text{var}_t]^2 & = c^2\sum_{t=1}^{T/2} \E \left [\frac{\text{var}_t^2}{(\hat{\sigma}_t + m)^2} \right ] \nonumber \\
	& \le c^2 \left ( \sum_{t=1}^{T/2} \E \left [\frac{\text{var}_t^2 - \hat{\sigma}_t^2
	}{(\hat{\sigma}_t + m)^2} \right ] + \sum_{t=1}^{T/2} \E \left [\frac{ \hat{\sigma}_t^2 
	}{(\hat{\sigma}_t + m)^2} \right ] \right ) \nonumber \\
	& \le c^2 \left ( \frac{1}{m^2} \sum_{t=1}^{T} \E \left [ | \text{var}_t^2 - \hat{\sigma}_t^2|
	\right ] + T  \right ) \nonumber \\
	& \le c^2 \left ( \frac{(M^2+D^2) T^{2/3}\ln(T^{2/3})}{m^2} + T  \right )  \le 3 c^2T
	\end{align}
	The last inequality follows by the choice on $m$. 
	Hence, from Eq. (\ref{eq:high proba}), we have  with probability at least $3/4$, 
	\begin{equation}\label{eq:high proba bound}
	 (\tsum_{t=1}^{T} \eta_t) (f(\overline{x}_T) -f^*) \le  4 (R^2 + 3c^2T )  
	\end{equation} 
	
	Next, by denoting $(x)_+ = \max(x,0)$, we lower bound the left hand side,
	\begin{align}
	\frac{1}{c} \sum \eta_t & = \sum \frac{1}{\hat{\sigma}_t +m} \nonumber \\
	& = \sum \frac{1}{\text{var}_t +m} + \sum \left ( \frac{1}{\hat{\sigma}_t +m} - \frac{1}{\text{var}_t +m} \right ) \nonumber \\
	& \ge \sum \frac{1}{\text{var}_t +m} - \sum  \frac{( \hat{\sigma}_t-\text{var}_t)_+ }{(\text{var}_t +m)(\hat{\sigma}_t +m)} \nonumber \\
	& \ge \sum \frac{1}{\text{var}_t +m} - \sum  \frac{( \hat{\sigma}_t-\text{var}_t)_+  }{ \sqrt{\text{var}_t+m} \cdot m^{3/2} } \nonumber \\
	& \ge \sum \frac{1}{\text{var}_t +m} - \sum  \frac{1}{2} \left (\frac{( \hat{\sigma}_t-\text{var}_t)_+^2  }{m^{3}}  + \frac{1}{\text{var}_t+m} \right )  \nonumber \\
	& = \frac{1}{2} \sum \frac{1}{\text{var}_t +m} - \frac{1}{2m^3} \sum  {(\hat{\sigma}_t-\text{var}_t)_+^2  } \\
	\end{align} 
	
	Note that 
    \begin{align*}
        \frac{1}{\text{var}_t + m} &=  \frac{1}{ \sqrt{\sigma_{2t}^2  + \sigma_{2t + 1}^2} / 2  + m} \ge \frac{1}{ \sigma_{2t}  + \sigma_{2t + 1}  + m} \\
        &\ge  \frac{1}{4} (\frac{1}{ \sigma_{2t} + m} + \frac{1}{ \sigma_{2t + 1}  + m} ) - \frac{1}{m^2} (\sigma_{2t}   - \sigma_{2t + 1}) \\
        &\ge  \frac{1}{4} (\frac{1}{ \sigma_{2t} + m} + \frac{1}{ \sigma_{2t + 1}  + m} ) - \frac{1}{m^2} \abs{\sigma_{2t}^2   - \sigma_{2t + 1}})
    \end{align*}
	
	Therefore, 
	\begin{align}\label{eq:lower bound left}
		\frac{1}{c} \sum \eta_t  \ge \frac{1}{2} \sum_k \frac{1}{\sigma_l +m} - \frac{1}{2m^3} \sum  {(\hat{\sigma}_t-\text{var}_t)_+^2 } - \frac{\sqrt{T}}{m^2} D 
	\end{align}
	
	Finally, by Markov's inequality, with probability $3/4$
	\[ \sum  {(\hat{\sigma}_t-\text{var}_t)_+^2  }  \le 4\E[\sum  {(\hat{\sigma}_k-\text{var}_t)_+^2  }] \le 4\E[\sum  {(\hat{\sigma}_k-\text{var}_t)^2  }]   \le 8(D^2 + M^2 ) T^{2/3} \ln (T^{2/3}). \] 
	Following the choice of 
	$m= 4\sqrt{D^2+M^2} T^{-\frac{1}{9}} \ln(T)^{\frac{1}{2}}$,
	we have 
	\[ \frac{1}{2m^3}\sum_{k=1}^T  {(\hat{\sigma}_k-\text{var}_t)_+^2  }  \le \frac{T}{4(M+m)} \le \frac{1}{4} \sum_{k=1}^T \frac{1}{\sigma_k +m}\] 
	\[
	\frac{\sqrt{T}}{m^2} D \le T/ 8(M+m) \le \frac{1}{8} \sum_{k=1}^T \frac{1}{\sigma_k +m}
	\]
	Consequently, together with (\ref{eq:high proba bound}) and (\ref{eq:lower bound left}), we know that with probability at least $1-\frac{1}{4} - \frac{1}{4} = 1/2$, 
	\begin{align}
	f(\overline{x}_T) - f^* \le \frac{4(R^2+3c^2T) }{\sum_k \frac{c}{8(\sigma_k + m)}} \le \frac{2R}{\sqrt{T}} \cdot \frac{64}{\sum_k \frac{1}{(\sigma_k + m)}},
	\end{align}
	where the last inequality follows by setting $c=\frac{R}{\sqrt{T}}.$
	
	
	
\end{proof}



\end{document}